\documentclass [12pt,reqno]{amsart}
\usepackage{amsmath}
\usepackage{amsthm}
\usepackage{dsfont}
\newtheorem{thm}{Theorem} 
\usepackage{amsfonts, amssymb}
\usepackage[utf8]{inputenc} 
\usepackage[english]{babel}
\usepackage{mathtools}
\usepackage{ mathrsfs }
\usepackage{enumerate}
\usepackage{ amssymb }
\usepackage{ bbold }
\usepackage{cite}

\newtheorem{theorem}{Theorem}

\newtheorem*{theorem*}{Theorem}

\newtheorem{lemma}[thm]{Lemma}

\theoremstyle{definition}
\newtheorem{remark}[thm]{Remark}

\theoremstyle{definition}

\theoremstyle{definition}
\newtheorem{definition}[thm]{Definition}

\numberwithin{equation}{section}

\def\R{{\mathbb R}}
\def\E{{\mathbb E\,}}

\def\Nu{{\mathcal V}}

\def\<{\langle}
\def\>{\rangle}

\def \ll {{\lambda}}

\begin{document}

\title{Scalar product and distance in barycentric coordinates}


\begin{abstract}
We give a formula for the scalar product and the distance 
in generalized barycentric coordinates for Euclidean and 
Pseudo-Euclidean spaces. 


\end{abstract}

\author{Vladimir Zolotov}
\address[Vladimir Zolotov]{Online poker player.}
\email[Vladimir Zolotov]{lemiranoitz@gmail.com}

\maketitle

\section{Difference from MO post}
There is a MathOverflow post \cite{MO433272} giving the same formula. 
The difference is that in the present paper,
we consider generalized barycentric coordinates
instead of the usual barycentric coordinates.
This results in a more complex proof. 
The discussion of connections to other results and
possible applications is also 
extended see Section \ref{RelApp}.

\section{Fixed objects}
Suppose we have points $x_1,\dots,x_n$ in an $m$-dimensional 
Euclidean (or Pseudo-Euclidean) space $(M, \<\cdot,\cdot\>)$ 
such that the affine span of $x_1,\dots,x_n$ is the whole $M$. 
Let's fix those $(M, \<\cdot,\cdot\>)$ and $x_1,\dots,x_n$ 
for the whole paper to avoid reintroducing 
them in every single statement.

For a vector $v \in M$ will use the notation $\vert \vert v \vert \vert^2 
:= \< v , v\>$. Note that in the case when $M$ is a 
Pseudo-Euclidean space
$\vert \vert v \vert \vert^2$ can be negative. 

We denote by $D$ the distance matrix for $x_1,\dots,x_n$, 
in other words, $D_{ij} = \vert \vert x_i - x_j \vert \vert^2$.

\section{Generalized barycentric coordinates}
\begin{definition}
For a point $p \in M$, we say that $a_1,\dots,a_n  \in \R$ are \textit{ generalized non-normalized barycentric coordinates} of $p$ iff 
$$a_1 + \dots + a_n \neq 0$$
and 
$$(a_1 + \dots + a_n)p = a_1 x_1 + \dots + a_n x_n.$$
We say that $a_1,\dots,a_n$ are \textit{ generalized absolute (or normalized) barycentric coordinates} of $p$ iff
$$a_1 + \dots + a_n = 1$$
and 
$$p = a_1 x_1 + \dots + a_n x_n.$$
\end{definition}

\begin{remark}
The world "generalized" means that $x_1,\dots,x_n$ do not have
to be in a general position. 
If they are in a general position 
then that's just the usual barycentric coordinates. 
In this case for the fixed point $p$ absolute barycentric
 coordinates are unique  
and non-normalized barycentric coordinates are
 unique up to a non-zero scaling factor. 
\end{remark}

Generalized barycentric coordinates have applications in finite
element methods, computer graphics, and the localization problem
for sensor networks.
For more information on generalized barycentric
coordinates see a survey \cite{floater_2015}. 

\section{Generalized displacement vector}
\begin{definition}
For a vector $v \in M$, we say that a vector
 $\bar v = (\bar v_1,\dots, \bar v_n) \in \mathbb{R}^n$ 
 is a \textit{generalized displacement vector} iff
there exist a point $p$ with generalized absolute  
barycentric coordinates $a_1,\dots,a_n$ 
such that the point $(v + p)$
has generalized absolute barycentric 
coordinates $a_1 + \bar  v_1, \dots, a_n + \bar  v_n$.
\end{definition}

\begin{remark}
Using that for generalized absolute barycentric 
coordinates $a_1,\dots,a_n$ for a point $p$ we have 
$$\sum_{i = 1}^{n}a_i = 1$$
we get that for any generalized displacement 
vector $\bar v = (\bar v_1,\dots, \bar v_n)$ we have
$$\sum_{i = 1}^{n}\bar v_i = 0.$$
\end{remark}

\begin{remark}
If $x_1,\dots,x_n$ are in general position (so our coordinates are not generalized) 
then for any $v$, there is a unique displacement vector $\bar v$.
Otherwise, there are multiple generalized displacement 
vectors for a single vector $v$.
\end{remark}

\section{The formula for the scalar product (and the distance)}
We remind that $D$ denotes the distance matrix 
for $x_1,\dots,x_n$ in other words 
$D_{ij} = \vert \vert x_i - x_j \vert \vert^2$.
\begin{theorem} \label{Thm}
Suppose we have vectors $u, v \in M$ and $\bar u$, $\bar v$
 are their generalized displacement vectors. Then 
\begin{equation} \<u, v\> = \bar u^{T} 
\left(-\tfrac{1}{2}D\right)\bar v.
  \label{BSP} \tag{BSP} \end{equation}

\end{theorem}
\begin{remark}
A trivial corollary is that  
$\vert \vert u \vert \vert^2 = \bar u^{T} \left(-\tfrac{1}{2}D\right)\bar u.$
\end{remark}

Before giving the proof of Theorem \ref{Thm} we establish 
a couple of lemmas providing more structure for the argument.
 
\begin{lemma} \label{L1}
Let $\ll_1,\dots,\ll_n \ge 0$ be such that 
$ \sum_{i = 0}^{n}\ll_i = 1$ and $y \in M$ then 
$$\sum_{i = 1}^{n} \ll_i \vert \vert y - x_i \vert \vert^2 =
 \vert \vert y - \sum_{i = 1}^n \ll_i x_i \vert \vert^2
 + \sum_{i = 1}^{n} \ll_i 
 \vert \vert x_i - \sum_{j = 1}^{n} \ll_j  x_j  \vert \vert^2.$$
\end{lemma}
\begin{remark}
The other way of stating this lemma is that for a point $p \in M$ and a random point $X$ supported on $x_1,\dots,x_n$ we have 
$$\E \vert \vert y - X \vert \vert^2 = \vert \vert y - \E X \vert \vert^2  + \E \vert \vert X - \E X \vert \vert^2.$$
\end{remark}
\begin{proof}[Proof of Lemma \ref{L1}]
We start from the left-hand side and eventually transform
 it into the right-hand side:
$$\sum_{i = 1}^{n} \ll_i \vert \vert y - x_i \vert \vert^2  = \sum_{i = 1}^{n} \ll_i \< y - x_i , y - x_i\> = $$
$$ = \sum_{i = 1}^{n} \ll_i \< (y - \sum_{j = 1}^{n} \ll_j  x_j) - (x_i - \sum_{j = 1}^{n} \ll_j  x_j) , (y - \sum_{j = 1}^{n} \ll_j  x_j) - (x_i - \sum_{j = 1}^{n} \ll_j  x_j)\> .$$
By applying $\<u - v, u - v\> = \vert \vert u \vert \vert ^ 2 +  \vert \vert v \vert \vert ^ 2 - 2\< u, v \>$ to the above it rewrites as 
$$ \vert \vert y - \sum_{j = 1}^n \ll_j x_j \vert \vert^2 + \sum_{i = 1}^{n} \ll_i \vert \vert x_i - \sum_{j = 1}^{n} \ll_j  x_j  \vert \vert^2 - 2 \sum_{i = 1}^{n} \ll_i \< y - \sum_{j = 1}^{n} \ll_j  x_j  ,  x_i - \sum_{j = 1}^{n} \ll_j  x_j\>.$$
This is the right-hand side of our formula, except the last summand which is extra. But that's not a problem because it is zero, indeed:
$$\sum_{i = 1}^{n} \ll_i \< y - \sum_{j = 1}^{n} \ll_j  x_j  ,  x_i - \sum_{j = 1}^{n} \ll_j  x_j\> = $$
$$= \< y - \sum_{j = 1}^{n} \ll_j  x_j  ,  \sum_{i = 1}^{n} \ll_i ( x_i - \sum_{j = 1}^{n} \ll_j  x_j)\> = $$ 
$$= \< y - \sum_{j = 1}^{n} \ll_j  x_j  ,  0\> = 0.$$ 

\end{proof}

\begin{lemma} \label{L2}
Let $\ll_1,\dots,\ll_n \ge 0$ and $\nu_1,\dots,\nu_n \ge 0$ be such that $\sum_{i = 0}^{n}\ll_i = 1$, $\sum_{i = 0}^{n}\nu_i = 1$.
Denote $ \Lambda = (\ll_1,\dots,\ll_n)^T$, $\Nu = (\nu_1,\dots,\nu_n)^T$ then 
$$\Lambda^T D \Nu = \sum_{i,j = 1}^{n} \ll_i \nu_j \vert \vert x_i - x_j \vert \vert^2 = $$
$$ = 
\sum_{i = 1}^{n} \ll_i \vert \vert x_i - \sum_{j = 1}^{n} \ll_j  x_j  \vert \vert^2 +
\vert \vert \sum_{j = 1}^{n} \ll_j  x_j - \sum_{j = 1}^{n} \nu_j  x_j \vert \vert^2 + 
\sum_{i = 1}^{n} \nu_i \vert \vert \sum_{j = 1}^{n} \nu_j  x_j - x_i \vert \vert^2.$$
\end{lemma}

\begin{remark}
This statement also can be presented via probability language: 
Let $X$ and $Y$ be two independent random variables supported on $x_1,\dots,x_n$. Then 
$$\E \vert \vert X - Y \vert \vert^2 = \E \vert \vert X - \E X \vert \vert^2  +  \vert \vert \E X - \E Y \vert \vert^2 + \E \vert \vert \E Y - Y \vert \vert^2.$$
\end{remark}

\begin{proof}[Proof of Lemma \ref{L2}]
We start from the left-hand side and apply Lemma \ref{L1}:
$$\sum_{i,j = 1}^{n} \ll_i \nu_j \vert \vert x_i - x_j \vert \vert^2 =
\sum_{i}^{n} \ll_i \sum_{j}^{n}  \nu_j \vert \vert x_i - x_j \vert \vert^2 = $$
$$
=\sum_{i}^{n} \ll_i \Big( \vert \vert x_i - \sum_{k = 1}^n \nu_k x_k \vert \vert^2 + \sum_{j = 1}^{n} \nu_j \vert \vert x_j - \sum_{k = 1}^{n} \nu_k  x_k  \vert \vert^2 \Big)
$$
$$
=\sum_{i}^{n} \ll_i \Big( \vert \vert x_i - \sum_{k = 1}^n \nu_k x_k \vert \vert^2  \Big) + \sum_{j = 1}^{n} \nu_j \vert \vert x_j - \sum_{k = 1}^{n} \nu_k  x_k  \vert \vert^2 .
$$
By applying Lemma 8 to the first summand, we get the right-hand side of the desired equality.
$$
\sum_{i}^{n} \ll_i  \vert \vert x_i - \sum_{k = 1}^n \ll_k x_k \vert \vert^2 + \vert \vert   \sum_{k = 1}^n \ll_k x_k - \sum_{k = 1}^n \nu_k x_k\vert \vert^2   + \sum_{j = 1}^{n} \nu_j \vert \vert x_j - \sum_{k = 1}^{n} \nu_j  x_k  \vert \vert^2.
$$
\end{proof}
\begin{proof}[Proof of Theorem \ref{Thm}]
Since $\bar u$ and $\bar v$ are generalized displacement vectors 
for $u$ and $v$, then there 
exist points $u_b$, $u_e = u_b + u$, $v_b$, $v_e = v_b + v$ with generalized absolute barycentric 
coordinates 
$$\ll_1, \dots, \ll_n,$$
$$\ll_1 + \bar u_1, \dots, \ll_n + \bar u_n,$$
$$\nu_1, \dots, \nu_n,$$
$$\nu_1 + \bar v_1, \dots, \nu_n + \bar v_n.$$

We denote 
$$\Lambda^b = (\ll_1, \dots, \ll_n)^T,$$ 
$$\Lambda^e = (\ll_1 + \bar u_1, \dots, \ll_n + \bar u_n)^T,$$
$$\Nu^b = (\nu_1, \dots, \nu_n)^T,$$ 
$$\Nu^e = (\nu_1 + \bar v_1, \dots, \nu_n + \bar v_n)^T,$$ and now 
$\Lambda^e = \Lambda^b + \bar u$, $\Nu^e = \Nu^b + \bar v$.
We start from the right-hand side of the desired equality and transform
it to the left-hand side:
$$\bar u^{T} \left(-\tfrac{1}{2}D\right)\bar u = 
(\Lambda^e - \Lambda^b)^T
\left(-\tfrac{1}{2}D\right)
(\Lambda^e - \Lambda^b) = $$
$$ = -\frac{1}{2}\big(\Lambda^e D \Nu^e + \Lambda^b D \Nu^b  -
\Lambda^e D \Nu^b - \Lambda^b D \Nu^e \big).$$
We apply Lemma \ref{L2} to each of the four summands. 
All the summands coming from the first and third parts 
of the right-hand side of Lemma \ref{L2} will cancel 
each other, and we get
$$ -\frac{1}{2}\big(\vert \vert u_e - v_e \vert \vert
 + \vert \vert u_b - v_b \vert \vert  -
\vert \vert u_e - v_b \vert \vert -
 \vert \vert u_b - v_e \vert \vert).$$
Now we apply $\vert \vert x - y \vert \vert =
\vert \vert x \vert \vert + 
\vert \vert y \vert \vert - 
2\< x , y\> $ to each of the four summands and get the left-hand side 
from Theorem \ref{Thm}.
$$\<u_e, v_e\> + \<u_b, v_b\> - \<u_e, v_b\>  - \<u_b, v_e\> = $$
$$\<u_e - u_b, v_e - v_b\> =  \< u, v\>.$$
\end{proof}
 
\section{Previous and similar results and possible applications} \label{RelApp}
\subsection{Previous results, guides for barycentric computations, and formula sheets.}
M. Schindler and E. Chen \cite{baryOlymp} provide a guide for barycentric 
computations in olympiad geometry. 
They only deal with $2$-dimensional case. 
There is no formula for the scalar product, but many of its corollaries
are given which they prove independently: condition for perpendicularity, length of a vector, circle equation, etc.
There is also a formula sheet for barycentric computations 
by G.Grozdev and D. Dekov \cite{grozdev2016barycentric}.

V. Volenec \cite{volenec2003metrical} gives a formula for the scalar 
product in the $2$-dimensional case. 
It is surprisingly different from what this paper provides.
V. Volenec also claims that 
his formula "implies all important formulas about metrical
relations of points and lines" and proceeds to recover those in the 
rest of the paper. 

It seems reasonable to follow Volenec's path 
and try to use Theorem \ref{Thm}
to generalize some of the results from \cite{baryOlymp, 
grozdev2016barycentric, volenec2003metrical} and maybe create some 
formula list for future reference. 
That being said if we want to deal 
with generalized barycentric coordinates and Pseudo-Euclidean spaces
and also provide details that can get busy for some formulas. 
So we are not going to do it in the present paper.

\subsection{Relationship between Gram matrix and distance matrix.}
Let $p = \frac{1}{n}\sum_{i=0}^{n}x_i$ and $G$ be a Gram matrix
 for vectors $x_1 - p, x_2 - p, \dots, x_n - p$, then 
$$
G = \left(I - \tfrac{1}{n}J\right)\left(-\tfrac{1}{2}D\right)\left(I - \tfrac{1}{n}J\right),
$$ 
where $J$ is an $n \times n$ matrix of all ones. 
This last formula is the key ingredient 
in the classical Multidimensional Scaling(MDS).
This formula can be seen as $n^2$ statements, 
one for each element of the Gram matrix.
Denote by $\mathbb{1}_n$ a vector of all ones.
Note that $(e_i -  \frac{1}{n}\mathbb{1}_n)$ is a generalized displacement vector for $(x_i - p)$. 
Thus one can see that those $n^2$ statements are precisely the 
formula from Theorem \ref{Thm} for $u,v \in \{x_1 - p, x_2 - p, \dots, x_n - p\}$.

\subsection{Negative type inequalities and embeddability 
in Euclidean spaces.}
Consider a metric space $Y = \{y_1,\dots,y_n\}$ with a distance 
matrix $D^Y$, i.e., $D^Y_{i,j} = d(y_i,y_j)$.
 $Y$ allows an isometric embedding into some Euclidean space iff 
 for every $\lambda \in \mathbb{R}^n$ such that 
 $\sum_{i = 1}^{n} \lambda_i = 0$ we have 
\begin{equation}\lambda^T D^Y \lambda \le 0. \label{NEG} \end{equation}
Those are called negative type inequalities, see Geometry of Cuts and Metrics
\cite{DL}.
  Let's identify points $y_1,\dots,y_n$ with some set of points
  in $\mathbb{R}^{n-1}$, which are in general position. 
  There exists a unique Pseudo-Euclidean scalar product on $\mathbb{R}^{n-1}$ 
  such that it agrees with the metric given by $D^Y$ 
  (for example, you can get it from Theorem \ref{Thm}).
  Theorem \ref{Thm} says that \eqref{NEG} condition 
  is a condition that the "squared length" 
  of each vector in $\mathbb{R}^{n-1}$ is non-negative.
  
\subsection{Sturm and Alexandrov spaces of non-negative curvature.}
Consider a metric space $Y = \{y_1,\dots,y_n\}$ with a distance 
matrix $D^Y$, i.e., $D^Y_{i,j} = d^2(y_i,y_j)$.
$Y$ has non-negative curvature in the sense of 
K.T. Sturm \cite{sturm1999metric} iff
 for every $\lambda \in \mathbb{R}^n$ such that 
  $\sum_{i = 1}^{n} \lambda_i = 0$, and only one of $\ll_i$ is negative 
  we have 
  \begin{equation}\lambda^T D^Y \lambda \le 0. \label{ST} \end{equation}
  Once again, let's identify $y_1,\dots,y_n$ with some set of points
  in $\mathbb{R}^{n-1}$, which are in general position. 
  And again, there exists a unique Pseudo-Euclidean 
  scalar product on $\mathbb{R}^{n-1}$, 
  such that it agrees with the metric given by $D^Y$.
Theorem \ref{Thm} says that $Y$ has non-negative curvature in sense of 
K.T. Sturm iff each vector going from a vertex to 
the opposite face has a non-negative "squared length". 

The above was previously noted by N. Lebedeva and A. Petrunin
 in the work \cite{lebedeva20225},
where they prove that $5$-point subsets of Alexandrov spaces
of non-negative 
curvature are exactly $5$-point Sturm spaces
 of non-negative curvature. See also 
\cite{petrunin2014quest}.

\subsection{Distance geometry and skipping realizations (embeddings).}
One natural application for distance formula in barycentric coordinates
comes from dealing with intrinsic triangulations 
 \cite{sharp2019navigating, gillespie2021integer,
 sharp2021geometry, sharp2021intrinsic}.
 In this setting, one has length of the edges of a triangle given to him but does not have 
a planar realization of this triangle readily available to him and interested 
in computing distances between points specified by their barycentric coordinates 
with respect to the triangle vertices. Although one can construct a planar realization,
compute barycenters and use the Pythagorean theorem, the distance formula in 
barycentric coordinates provides a much cleaner path. 

Distance geometry (see \cite{liberti2014euclidean}) is a field of 
geometry dealing with the setting described 
above. Namely, all the information available is given as some distances, 
and in particular, the Euclidean realization is not given. 
Since Theorem \ref{Thm}
only relies on distances, it may be useful in this field to 
do some geometry prior to or without constructing a realization. 


\subsection*{Acknowledgements}
I thank Nina Lebedeva for valuable discussions and for finding bugs in
an early version of this paper.  

\bibliography{circle}

\bibliographystyle{plain}

\end{document}